\documentclass[11pt]{article}
\usepackage{amsmath, amsthm, amscd, amsfonts, amssymb, graphicx, color}
\usepackage[margin=2.8cm]{geometry}
\usepackage[bookmarksnumbered, colorlinks, plainpages]{hyperref}

\newtheorem{theorem}{Theorem}[section]
\newtheorem{lemma}[theorem]{Lemma}
\newtheorem{proposition}[theorem]{Proposition}
\newtheorem{corollary}[theorem]{Corollary}
\theoremstyle{definition}
\newtheorem{definition}[theorem]{Definition}

\theoremstyle{remark}

\numberwithin{equation}{section}

\DeclareMathOperator{\supp}{supp}

\allowdisplaybreaks
\begin{document}

\title{\textbf{\emph{S}-Noetherian generalized power series rings}}

\author {\bf  F. Padashnik, A. Moussavi and H. Mousavi}
\date{}
\maketitle
\begin{center}

{\small Department of Pure Mathematics, Faculty of Mathematical
Sciences,\\ Tarbiat Modares University, Tehran, Iran, P.O. Box:
14115-134.{\footnote {\noindent Corresponding author.
moussavi.a@modares.ac.ir and  moussavi.a@gmail.com.\\
\indent f.padashnik@modares.ac.ir \\
\indent h.moosavi@modares.ac.ir .}}
}\\
\end{center}

\date{}
\maketitle

\begin{abstract}
Let  $R$ be a ring with identity, $(M,\leq)$  a commutative positive strictly ordered monoid and $\omega_{m}$  an automorphism for each $m\in M$.   The skew generalized power series ring $R[[M,\omega]]$ is a common generalization of (skew) polynomial rings, (skew)
power series rings, (skew) Laurent polynomial rings, (skew) group rings, and Mal'cev Neumann Laurent  series rings. If $S\subset R$ is a  multiplicative set, then $R$ is called right $S$-\emph{Noetherian}, if for each ideal $I$ of $R$, $Is\subseteq J \subseteq I$ for some $s \in S$ and some finitely generated right ideal $J$. Unifying and generalizing a number of known results, we study
transfers of $S$-Noetherian property to the  ring $R[[M,\omega]]$.
We also show that the ring $R[[M,\omega]]$
is left Noetherian if and only if $R$ is left Noetherian and $M$ is finitely generated.  Generalizing a result of Anderson and Dumitrescu, we show that,when $S\subseteq R$ is a $\sigma$-anti-Archimedean multiplicative set with $\sigma$  an automorphism of $R$, then $R$ is right  $S$-Noetherian  if and only if the skew polynomial ring $R[x;\sigma]$ is  right $S$-Noetherian.
\end{abstract}



\section{Introduction}

Throughout this paper, $R$ is a ring (not necessary commutative) with identity.
In \cite{anderson KZ}, the authors introduced the concept of ``almost
finitely generated" to study Querr$\acute{e}$'s characterization of divisorial ideals in integrally
closed polynomial rings. Later, Anderson and Dumitrescu \cite{anderson} abstracted this notion
to any commutative ring and defined a general concept of Noetherian rings. They
call $R$ an $S$-Noetherian ring if each ideal of $R$ is $S$-finite, i.e., for each ideal $I$ of $R$,
there exist an $s\in S$ and a finitely generated ideal $J$ of $R$ such that $Is\subseteq J \subseteq I$. By \cite[
Proposition 2(a)]{anderson}, any integral domain $R$ is ($R \setminus \{0\}$)-Noetherian; so an $S$-Noetherian
ring is not generally Noetherian. Also, $M$ is said to be $S$-finite if there exist an $s\in S$
and a finitely generated $R$-submodule $F$ of $M$ such that $sM \subseteq F.$ Also, $M$ is called
$S$-Noetherian if each submodule of $M$ is $S$-finite.
In \cite{anderson}, the authors gave a number of $S$-variants of well-known results for
Noetherian rings: $S$-versions of Cohen's result, the Eakin-Nagata theorem, and
the Hilbert basis theorem under an additional condition. More precisely, in \cite[
Propositions 9 and 10]{anderson}, the authors showed that, if $S$ is an anti-Archimedean
subset of an $S$-Noetherian ring $R$, then the polynomial ring $R[X_1,\cdots,X_n]$ is also
an $S$-Noetherian ring; and if $S$ is an anti-Archimedean subset of an $S$-Noetherian
ring $R$ consisting of nonzero divisors, then the power series ring $R[[X_1,\cdots,X_n]]$ is
an $S$-Noetherian ring. Note that if $S$ is a set of units of $R$, then the results above
are nothing but the Hilbert basis theorem and a well-known fact that $R[[X]]$ is
Noetherian if $R$ is Noetherian. In \cite[Theorem 2.3]{liu}, Liu generalized this result to the
ring of generalized power series as follows: If $S$ is an anti-Archimedean subset of a
ring $R$ consisting of nonzero divisors and 	($\Gamma,\leq$)
 is a positive strictly ordered monoid
(defined in Secion 4), then $R[[M,\le]]$ is $S$-Noetherian if and only if $R$ is $S$-Noetherian
and $\Gamma$ is finitely generated. Note that this recovers the result for the Noetherian case
shown in \cite[, Theorem 4.3]{brook} when $S$ is a set of units.
Also, the authors in \cite{lim ex} study on transfers of  the $S$-Noetherian property to the constructions
$D+(X_1,\cdots,X_n)E[X_1,\cdots,X_n]$ and $D+(X_1,\cdots,X_n)E[[X_1,\cdots,X_n]]$
and Nagata's idealization is studied in \cite{lim amal}.\par

The authors in \cite[Theorem 7.7, page(65)]{Gilmer} proved that
$R[M]$ is Noetherian if and only if $R$ is Noetherian and $M$ is finitely generated.
 Brookfield \cite{brook} proved  that if $(M,\le)$ is a commutative positively ordered monoid,
then $R[[M,\le]]$ is right Noetherian if and only if $R$ is right Noetherian and $M$ is finitely generated.\par

Ribenboim  \cite{Ribenboim-noeth}
  and
 Varadarajan \cite{vara},
   have carried out an extensive study of
rings of generalized power series. They investigated conditions under which a
ring of generalized power series $R[[M,\le]]$ is Noetherian, where $R$ is a commutative
ring with identity and $(M,\le)$ is a strictly ordered monoid.\par

In this paper we obtain results pertaining to Noetherian nature of
generalized power series rings. These considerably strengthen earlier
results of Ribenboim \cite{Ribenboim-noeth}, Varadarajan \cite{vara}, Brookfield \cite{brook}, D. D. Anderson, and T. Dumitrescu\cite{anderson}, D. D. Anderson, B. G. Kang, and  M. H. Park \cite{andersonI} , D. D. Anderson, D. J. Kwak, M. Zafrullah \cite{anderson KZ} on this topic.\par

More precisely, we show that, if $S$ is an $\sigma$-anti-Archimedean multiplicative
subset of an $S$-Noetherian ring $R$ with an automorphism $\sigma$, then the skew polynomial ring $R[x;\sigma]$ is also
an $S$-Noetherian ring; and if $(M,\le)$ is a commutative positively ordered monoid and $\omega_m$
is an automorphism over $R$ for every $m\in M$, then the skew generalized power series ring $R[[M,\omega]]$ is right Noetherian if and only if $R$ is right Noetherian and $M$ is finitely generated.
When $(M,\leq)$ is a commutative positive strictly ordered monoid and $\omega_{m}$ is an automorphism for each $m\in M$,
we unify and generalize the above mentioned results, and study
transfers of $S$-Noetherian property to the skew generalized power series ring $R[[M,\omega]]$.

\section{S-Noetherian property on skew polynomial rings}

If $R$ is a commutative ring
and $S$ is a multiplicative subset of $R$,
in \cite{anderson}, the authors proved  that the necessary condition
for the ring of fractions $R_S$ to be a Noetherian ring is that $R$ be an
$S$-Noetherian ring. In noncommutative rings, the
situation is  more complicated.  In fact, if $S$ is a right (resp., left)
permutable and right (resp., left) reversible (i.e $S$ is
right (resp., left) denominator set), then $R$ has a ring
of fraction $RS^{-1}$ (resp., $S^{-1}R$). In this situation,
denominator sets (both left and right denominator sets) act
like a multiplicatively closed sets in the commutative case. Our
interest in this note is multiplicatively closed subsets (i.e.
denominator subsets) in noncommutative rings.
 First we
define the notion of $S$-Noetherian rings for noncommutative rings.

\begin{definition}
Let $R$ be a ring and $S$ a multiplicative subset
of $R$. An ideal $I$ of $R$ is called right $S$-\emph{finite} (resp., $S$-\emph{principal}),
if there exists a finitely generated (resp., principal) right ideal $J$ of
$R$ and some $s\in S$ such that $Is\subseteq J \subseteq I$.\\

A ring $R$ is said to be right $S$-\emph{Noetherian} (resp., $S$-\emph{PRIR}),
if each right ideal of $R$ is right $S$-\emph{finite} (resp., $S$-\emph{principal}).
This definition can be done similarly for left side ideals.\\

Also, we say that an $R$-module $M$ is right (or left)
$S$-\emph{finite }if $Ms\subseteq F$ (resp., $sM\subseteq F$) for some
$s\in S$ and a finitely generated submodule $F$ of $M$.
A module $M$ is called right (or left) $S$-\emph{Noetherian}
if each submodule of $M$ is
a right (or left) $S$-finite module.
\end{definition}

The author in \cite{anderson} justified the definition of $S$-Noetherian for commutative rings
by proving some interesting properties of $S$-Noetherian ring. For
example, they showed that if $R$ is $S$-Noetherian, then the ring of fractions $R_S$ is
Noetherian and they found the conditions for the reverse of this proposition.\\

 Given rings $R,T$,  an ideal $J$ of $T$
is said to be extended, if there exists an ideal $I$ of $R$ such that $\varphi(I)=J$
where $\varphi:R\longrightarrow T$ is a ring monomorphism. Also,
a ring $R$ is von Neumann regular if for every $a \in R$ there
exists an $x$ in $R$ such that $a=axa$. The center of a ring $R$ is denoted
by $Cent(R)$.

\begin{proposition}\label{tie}
Let $R$ be a ring, $S\subseteq R$  a multiplicative  set
and $I$ a right ideal of $R$.

\emph{1}) If $R$ is  von Neumman regular, $S$  a denominator set  and $I\cap S\neq \emptyset$,
then $I$ is right $S$-principal.

\emph{2}) If $S\subseteq T$ are right denominator subsets of $R$
and $R$ is right $S$-Noetherian(resp., $S$-PRIR), then
$R$ is right $T$-Noetherian(resp., $T$-PRIR).

\emph{3}) If $R$ is  von Neumman regular and $S$  a denominator
set, then $R$ is right $S$-Noetherian (resp., $S$-PRIR) if and
only if $R$ is right Noetherian (resp., PRIR).

\emph{4}) If $S$ is a denominator set  and $R$ is right $S$-Noetherian (resp., $S$-PRIR),
 then $RS^{-1}$ is right Noetherian.

\emph{5}) If $S$ is central in $R$, then the conditions \emph{1-4} and those of \cite[Proposition \emph{2}]{anderson} follow.
\end{proposition}

\begin{proof}
1) Let $S\subseteq R$ be a denominator set, $R$  a von Neumman regular ring and $I$  a right ideal of $R$.
Then for each $s\in I\cap S$, one can see that $Is\subseteq Rs=s\frac{1}{s}Rs$,
where $\frac{1}{s}$ is the inverse of $s$ in $RS^{-1}$. It is sufficient to see
that $\frac{1}{s}Rs\subseteq R$.
For each $s\in S$, there exists $a\in R$ such that $sas=s$,
so $sa=s\frac{1}{s}=1$ (in $RS^{-1}$). Thus $sa=1$ and hence $a=\frac{1}{s}$.
Therefore $\frac{1}{s}\in R$ and $Rs\subseteq R$, so $\frac{1}{s}Rs \subseteq R$.

2) Let $S\subseteq T$ be denominator subsets of $R$. If $R$ is right
$S$-Noetherian (resp., $S$-PRIR), then for each right ideal of $R$, there exists
$s\in S$ such that $Is\subseteq J \subseteq I$ for some finitely generated (resp.,
principal) right ideal of $R$. Since $s\in S$, $S\subseteq T$,  $s\in T$ which
means that $R$ is right $T$-Noetherian (resp., $T$-PRIR).

3) Assume that $R$ is a right Noetherian (resp., PRIR) ring. Each right ideal of $R$ is
finitely generated (resp., principal). So for each $s\in S$, one can see that
$Is\subseteq I$. Hence $R$ is right $S$-Noetherian (resp., $S$-PRIR).
On the other hand, assume that $R$ is right $S$-Noetherian (resp., $S$-PRIR),
so there exists $s\in S$ such that $Is\subseteq J\subseteq I$ for some finitely
generated (resp., principal) right ideal of $R$. Also suppose that $sts=s$ for some
$t\in R$. So $Is\subseteq I$. Also, $It\subseteq I$, so $Its\subseteq Is=Ists$. So
$Its.\frac{1}{s}\subseteq Ists.\frac{1}{s}$. Hence $It\subseteq Ist$. Also $Is\subseteq I$
yields that $Ist\subseteq It \subseteq Ist$. So $Ist=It$. Thus $Ists=Its$ which
means that $Is=Ists=Its$. However $Its=I\frac{1}{s}sts=I\frac{1}{s}s=I$. So
$Is=I$. Thus $I=Is\subseteq J \subseteq I$ and hence $I=J$, and since $J$ is
a finitely generated (resp., principal) right ideal of $R$, so is $I$.

4) This proof is an inspiration from \cite[proposition 3.11 part (i)]{atia}.
First, we claim that each ideal of $RS^{-1}$ is extended. Let a right ideal $J$ of ring of fraction $RS^{-1}$
and $\frac{x}{s}=b\in J$. So $\frac{x}{1}=\frac{x}{s}.\frac{s}{1}\in J.\frac{s}{1}\subseteq J$.
So $\frac{x}{1}\in J$. Hence, $\varphi^{-1}(\frac{x}{1})\in \varphi^{-1}(J)$
which means that $x\in \varphi^{-1}(J)$. Thus, $\varphi(x)\in \varphi(\varphi^{-1}(J))$,
so $\frac{x}{1}\in \varphi(\varphi^{-1}(J))$. So $\frac{x}{1}.\frac{s}{s}=\frac{x.s}{s}.\frac{s}{s}=
\frac{xs}{s}\in \varphi(\varphi^{-1}(J))$. Note that $\varphi(\varphi^{-1}(J))$ is an ideal of $RS^{-1}$
and $s\in U(RS^{-1})$, so we have
\begin{align*}
\frac{xs}{s}.\frac{1}{s}=\frac{x}{s}\in \varphi(\varphi^{-1}(J))\frac{1}{s}\subseteq \varphi(\varphi^{-1}(J)).
\end{align*}
So $b=\frac{x}{s}\in \varphi(\varphi^{-1}(J))$ which implies $J\subseteq \varphi(\varphi^{-1}(J))$.
On the other hand, $\varphi(\varphi^{-1}(J))\subseteq J$ holds for each ideal of $RS^{-1}$. Thus
$J=\varphi(\varphi^{-1}(J))$ and $J$ is an extended ideal of $RS^{-1}$.

Let a right ideal $K$ of ring of fraction $RS^{-1}$. Since $R$ is right $S$-Noetherian there exists $s\in S$
and a finitely generated (resp., principal) right ideal $W$ of $R$ such that $\varphi^{-1}(K)s\subseteq W\subseteq \varphi^{-1}(K)$.
So $\varphi(\varphi^{-1}(K)s)\subseteq \varphi(W)\subseteq \varphi(\varphi^{-1}(K))$.
We know that $\varphi(\varphi^{-1}(K)s)=\varphi(\varphi^{-1}(K))\varphi(s)$.
Also, $\varphi(s)\in U(RS^{-1})$ and $\varphi(\varphi^{-1}(K))=K$. So $K\subseteq \varphi(W) \subseteq K$.
So $K=\varphi(W)$. Since $W$ is finitely generated,  $\varphi(W)$ is finitely generated.
So $K$ is finitely generated which means that $RS^{-1}$ is right Noetherian.

5) The proof is straightforward by \cite[Proposition 2]{anderson}.
\end{proof}

Now we generalize  a theorem of D.D. Anderson and Tiberiu Dumitrescu \cite[Proposition 9]{anderson}, for commutative polynomial ring $R[x]$, in a more general setting.
We show that  if $R$ is a right (or left) $S$-Noetherian ring with  an  automorphism $\sigma$, then $R[x;\sigma]$ is a right (or left) $S$-Noetherian ring.

In \cite{andersonI} the authors defined the notion of anti-Archimedean multiplication set. Now  we introduce the notion of $\sigma$-anti-Archimedean multiplication set:

\begin{definition}
Let $R$ be a ring with   an automorphism $\sigma$ and  $S$
a multiplicative set. Then $R$ is called left $\sigma$-\emph{anti}-\emph{Archimedean} over
$S$, if there exists $s\in S$, such that
\begin{align*}
(\bigcap_{l\geq 1,k_i\ge 0} R\sigma^{k_1}(s)\sigma^{k_2}(s)\cdots \sigma^{k_l}(s)\big)\cap S \neq \emptyset.
\end{align*}
\end{definition}

\begin{theorem}\label{S-Noetherian}
Let $R$ be a ring with an automorphism $\sigma$ and $S\subseteq R$  a $\sigma$-anti-Archimedean multiplicative set. Then $R$ is right \emph{(}or left\emph{)} $S$-Noetherian  if and only if $R[x;\sigma]$ is  right \emph{(}or left\emph{)} $S$-Noetherian.
\end{theorem}

\begin{proof}
($\Rightarrow$) We prove the theorem for the right version. The proof of left version is similar. First, we claim that if $D$ is a finitely generated $R$-module and $R$ is a right $S$-Noetherian ring, then $D$ is a right $S$-Noetherian module. For this claim, assume that $D$ is a finitely generated right $R$-module. So there exists a finitely generated free right $R$-module $F$ and a surjective homomorphism $\pi:F\longrightarrow D$. We show that $D$ is a right $S$-Noetherian $R$-module. For this, let  $N:=\pi^{-1}(T)$, for a submodule $T$ of $D$. We have $N\simeq I_1\oplus I_2 \cdots \oplus I_l$, for some right ideals $I_i$ of $R$, $1\leq i \leq l$. Since $R$ is a right $S$-Noetherian ring, there exists $s_i\in S$ such that $I_is_i\subseteq J_i$
 for a finitely generated ideals $J_i$ of $R$, $1\leq i \leq l$. Now take $s':=s_1s_2\cdots  s_l \in S$,  we show that $Ns'\subseteq K$ for  a finitely generated $R$-submodule $K$ of $F$. One can see that $Ns_1=I_1s_1\oplus I_2s_1\oplus \cdots \oplus I_ls_1$. Since $I_i$ is a right ideal of $R$ so we have $I_is_1\subseteq I_i$ for $i\neq 1$ and $I_1s_1\subseteq J_1$, for a finitely generated right ideal $J_1$ of $R$. So we have $Ns_1\subseteq J_1\oplus I_2\oplus I_3 \cdots \oplus I_l$. Continuing in this way,
 $Ns_1s_2\cdots s_l\subseteq J_1 \oplus J_2\oplus \cdots \oplus J_l\simeq K$,
 where $J_i$ is a finitely generated right ideal of $R$,$1\leq i \leq l$, and hence $K$ is a finitely generated $R$-submodule of $F$. Thus $Ns'\subseteq K$ and hence $F$ is a right $S$-Noetherian $R$-module. Next, since $T=\pi(N)$ and $Ns'\subseteq K$, we have $\pi(Ns')=\pi(N)s'=Ts'\subseteq \pi(K)$.
  We know that $K$ is finitely generated in $F$, so $\pi(K)$ is finitely generated $R$-submodule of $D$. Thus, $Ts'\subseteq \pi(K)$ which means that $T$ is $S$-finite. Since $T$ is an arbitrary $R$-submodule of $D$,  $D$ is a right $S$-Noetherian module.

Now, we  prove that $A:=R[x;\sigma]$ is a right $S$-Noetherian ring. Let $I$ be right ideal of $A$ and suppose that
\begin{align*}
J=\{r_i\in R | r_i\,\, \textrm{is a leading coefficient of any polynomial in} \,\, I\}\cup \{0\}.
\end{align*}
It is easy to see that $J$ is a right ideal. Since $R$ is right $S$-Noetherian, $Js\subseteq (a_1R+a_2R+\cdots+a_nR)$ for some $s\in S$ and $a_i \in J$. So there exist polynomials $f_i\in I$ with $f_i=a_{i,n_i}x^{n_i}+\cdots +a_{0,i}$. Let $d=max\{n_i\}$.
Assume that $T$ is the set of all polynomials in $I$ with degree less than $d$. Obviously, $T$ is a finitely generated right $R$-submodule of $A$. So by the first claim, $T$ is right $S$-Noetherian. Hence there exist $t \in S$, $g_i \in T$ for $1\le i \le m$ such that
$Tt\subseteq (g_1R+g_2R+\cdots +g_mR)$. Let $h(x)=\sum_{i=1}^zb_ix^i \in I$, so $b_z\in J$ which means that
$b_z\in (a_1R+a_2R+\cdots+a_nR)$. Thus $h\sigma^{-z}(s)$ can be written as follows:
\begin{align*}
h\sigma^{-z}(s)= v^{(1)}+w^{(1)}+q^{(1)},
\end{align*}
where $v^{(1)}\in (f_1A+f_2A+\cdots+f_nA)$, $w^{(1)}\in \{f\in A | d+1\le deg(f)\le z-1\}$ and $q^{(1)}\in T$. Continuing in this way and multiplying $\sigma^{-z+1}(s),\sigma^{-z+2}(s),\cdots,\sigma^{-1-d}(s)$  from right side  respectively, so there exists some $v\in (f_1A+f_2A+\cdots+f_nA)$,
$w\in T$ such that
\begin{align*}
h\sigma^{-z}(s)\sigma^{-z+1}(s)\cdots \sigma^{-d-1}(s)=v+w.
\end{align*}
Assume that $s_i=\sigma^{-z+i}$ and multipling $t$ from right side, then $ hs_1s_2\cdots s_{z-d}t =vt+wt$. But $wt\in Tt$, so
$wt\in (g_1R+g_2R+ \cdots+g_mR) \subseteq (g_1A+g_2A+\cdots+g_mA)$. Hence,
\begin{align*}
hs_1s_2\cdots s_{z-d}t\in (f_1A+f_2A+\cdots +f_nA+g_1A+\cdots +g_mA).
\end{align*}
Since $s_i$'s and $t$ are independent from the choice of $h\in I$, we have
\begin{align*}
Is_1s_2\cdots s_{z-d}t\subseteq (f_1A+f_2A+\cdots+ f_nA+g_1A+\cdots +g_mA).
\end{align*}
Finally, since $s_1s_2\cdots s_{z-d}t \in S$, the ideal $I$ is $S$-finite and because $I$ was chosen an arbitrary right ideal of $A$, hence $A$ is a right $S$-Noetherian ring.

($\Leftarrow$)
Let $I$ be a right ideal of $ R$. Suppose that
\begin{align*}
J=\{f\in A | \,\, \textrm{the leading coefficient of}\, f\, \textrm{is in}\, I\}.
\end{align*}
Then $J$ is a right ideal of $A$. Since $A$ is right $S$-Noetherian,
there exists $s\in S$ such that $Js\subseteq K\subseteq J$,
where $K$ is a finitely generated right ideal of $A$. Suppose that $K=(f_1A+f_2A+\cdots +f_lA)$.
Let $r\in I$, then there exists some $f\in J$ such that $fs=\sum a_if_i$.
So if $r_i$ is the leading coefficient of $f_i$, $1\leq i\leq l$, then
$rs\in (r_1R+r_2R\cdots +r_lR)$. So $Is\subseteq (r_1R+r_2R+ \cdots +r_lR)$.
Also, $K\subseteq J$, so each leading coefficient of $K$ is a leading coefficient of $J$.
So $(r_1R+r_2R+ \cdots +r_lR) \subseteq I$ and hence $I$ is right $S$-finite and $R$
is right $S$-Noetherian.
\end{proof}
We have the following generalization of a theorem of D.D. Anderson and Tiberiu Dumitrescu \cite[Proposition 9]{anderson}.

\begin{corollary}
Let $R$ be a \emph{(}not necessarily commutative\emph{)} ring and $S\subseteq R$ an anti-Archimedean multiplicative set.
If $R$ is $S$-Noetherian then so is the polynomial ring $R[X_1,X_2,\cdots,X_n]$.
\end{corollary}

\section{Noetherian Skew Generalized Power Series Rings}
Throughout this section, $(M,\leq)$ is assumed to be a strictly ordered commutative monoid.
 The pair $(M,\le)$ is called an \textit{ordered monoid} with
order $\le$, if for every $m,m', n\in M$, $m\le m'$ implies that
$nm\le nm'$ and $mn\le m'n$. Also, an ordered monoid $(M,\le)$ is
said to be \textit{strictly orderd} if for every $m,m',n\in M$, $m<m'$
implies that $nm<nm'$ and $mn<m'n$.
Let $(M,\le)$ be a partially ordered set. The set $(M,\le)$ is called \textit{Artinian}
if every strictly decreasing sequence of elements of $M$ stablized, and
also $(M,\le)$ is called \textit{narrow} if the number of
incomparable elements in every subset of $M$  is finite. Thus, we can conclude that $(M, \leq )$ is
Artinian and narrow if and only if every nonempty subset of $M$ has
at least one but only a finite number of minimal elements.

The author in \cite{Ribenboim-semisimple} introduced the ring of
generalized power series $R[[M]]$ for a strictly ordered monoid $M$
and a ring $R$ consisting of all functions from $M$ to $R$ whose support
is Artinian and narrow with the pointwise addition and the
convolution multiplication. There are a lot of interesting
examples of rings in this form (e.g., Elliott and Ribenboim,
\cite{Elliott-Ribenboim 1992}; Ribenboim,\cite{Ribenboim (1995b)}) and it was extensively studied by
many authors, recently.

In \cite{zim}, the authors defined a  ``twisted'' version of the mentioned
construction and study on ascending chain condition for its principal ideals.
Now we recall the construction of the skew generalized power series
ring introduced in \cite{zim}. Let $R$ be a ring, $(M, \leq )$ a
strictly ordered monoid, and $\omega: M\rightarrow End(R)$ a monoid
homomorphism. For $m \in M$, let $\omega_{m}$ denote the image of
$m$ under $\omega$, that is $\omega_{m}=\omega (m)$. Let $A$ be the
set of all functions $f: M\rightarrow R$ such that the support
$\supp(f)= \{ m\in  M | f(m)\neq 0 \}$ is Artinian and narrow. Then
for any $m \in M$ and $f,g \in A $ the set
\begin{align*}
\chi_{m}( f, g)=\{(u,v) \in \supp(f) \times \supp(g):  m=uv\}
\end{align*}
is finite. Thus one can define the product $f g: M\rightarrow R$ of
$f, g \in A$ as follows:
\begin{align*}
fg(m)=\sum_{ (u,v)\in \chi_{m}(f,g)}f(u)\omega_{u} (g(v)),
\end{align*}
(by convention, a sum over the empty set is $0$). Now, the set $A$
 with pointwise addition and the defined multiplication is a ring, and
called \textit{the ring of skew generalized power series} with
coefficients in $R$ and exponents in $M$. To simplify, take $A$ as a formal series $\sum\limits_{m\in M}r_mx^m,$
where $r_m = f(m)\in R$. This ring can be denoted either by
$R[[M^{\leq},\omega]]$ or by $R[[M,\omega ]]$ (see \cite{unified}
and \cite{von}).\par

For every $r\in R$ and $m\in M$ we can defined the maps $c_r,e_m:M\longrightarrow R$ by
\begin{align}\label{e_m}
c_r(x)=\begin{cases}
r\quad; x=1 \\
0 \quad ; \text{Otherwise}
\end{cases},
e_m(x)=\begin{cases}
1\quad; x=m \\
0 \quad  ; \text{Otherwise}
\end{cases}
\end{align}
where $x\in M$.
By way of illustration, $c_r(x)$ and $e_m(x)$ are like $r$ and $x^m$ in usual polynomial ring $R[x]$, respectively.\\

The following proposition which is proved in \cite[Theorem 2.1]{higman}, can characterize all Artinian and narrow sets.
\begin{proposition}
Let $(M,\leq)$ be an ordered set. Then the following conditions are
equivalent

\emph{(1)} $(M,\leq)$ is Artinian and narrow.

\emph{(2)} For any sequence $(m_n)_{n\in \mathbb{N}}$ of elements of $M$ there exist indices $n_1<n_2<n_3<\cdots$ such that $m_{n_1}\leq m_{n_2}\leq m_{n_3}\leq \cdots$ .

\emph{(3)} For any sequence $(m_n)_{n\in N}$ of elements of $M$ there exist indices $i < j$ such that $m_i\leq m_j$.
\end{proposition}
The author in \cite{brook} introduced the concept of a lower set. A \emph{lower set} of $L$ is a subset $I\subseteq L$ such that $x\leq y \in I$ implies $x \in I$ for all $x, y \in L$, (which we denoted by $\Downarrow L$ for the set of lower sets of $L$ ordered by inclusion). In this concept, we can ignore the condition narrow by lower set, indeed it is proved that if $L$ is a partially ordered set, then $\Downarrow L$ is Artinian if and only if $L$ is Artinian and narrow.  He also showed that if $\alpha : K \longrightarrow L$ is strictly increasing map between partially ordered sets, then if $L$ satisfies Artinian (or Noetherian) property, then so is $K$. Moreover, if $\alpha$ is surjective and $\Downarrow K$ satisfies Artinian (or Noetherian) property, then so does $\Downarrow L$.

An ordered monoid $(M,\leq)$ is called \emph{positively ordered} if $m\ge 0$ for all $m\in M$. In this condition,
$m\preceq m'$ implies $m \leq m'$ for all $m, m' \in M$. Now, according to \cite[in section 4]{brook} we have
\begin{align}
R[[M,\omega,\leq]]= \{f\in R[[M,\omega]] \quad|  \Downarrow (\supp(f),\leq)\quad is \quad \text{Artinian}\}.
\end{align}
If $\Downarrow (M,\leq)$ is Artinian, $R[[M,\omega,\leq]]=R[[M,\omega]]$.
For instance, $\Downarrow (\mathbb{F},\preccurlyeq)$
 and $\Downarrow (\mathbb{F}^n,\preccurlyeq)$ are
Artinian, and so $R[[\mathbb{F},\omega,\preccurlyeq]]=R[[\mathbb{F},\omega]]$
and $R[[\mathbb{F}^n,\omega,\preccurlyeq]]=R[[\mathbb{F}^n,\omega]]$
such that $\mathbb{F}$ be a free monoid.\\

Now we give a generalization of a result  \cite[Theorem 4.3]{brook} of G. Brookfield:
\begin{theorem}\label{go}
Let $R$ be a ring, $(M,\leq)$ a positive strictly ordered monoid and $\omega_{m}$  an automorphism of $R$  with $\omega_m\omega_n=\omega_n\omega_m$ for each $m,n\in M$.
Then $R[[M,\omega]]$ is  left Noetherian if and only if $R$ is left Noetherian and $M$ is finitely generated.
\end{theorem}
\begin{proof}
$\Leftarrow$) In the first place, we claim that if $\varphi :(N,\leq)\rightarrow (M,\leq)$ is a surjective strict monoid homomorphism, induces a surjective ring homomorphism $\varphi^*: R[[N,\omega,\leq]]\rightarrow R[[M,\omega,\leq]]$.
Since $\varphi$ is strict, $\varphi^{-1}(x)$ is antichain in $(N,\leq)$ for all $x\in M$. Thus, if
$f\in R[[N,\omega,\leq]]$ then $\varphi^{-1}(x)\cap \supp(f)$ is finite and we can define $\varphi^*(f)=f^*$, where
$f^*(x)= \sum_{x' \in \varphi^{-1}(x)} f(x')$ for $x\in M$.
We show that $\varphi^*$ is a ring homomorphism. One can see that
\begin{align}\label{A}
(fg)^*(m)=\sum_{m'\in \varphi^{-1}(m)}(fg)(m')=\sum_{xy=m}\sum_{\substack{x'y'=m'\\m'\in \varphi^{-1}(m)}}\bigg(f(x')\alpha_{x'}(g(y'))\bigg).
\end{align}
On the other hand
\begin{align*}
(f^*g^*)(m)=&\big(\varphi^*(f)\varphi^*(g)\big)(m)=\sum_{xy=m}\bigg(\varphi^*(f(x))\alpha_x\big(\varphi^*(g(y)\big)\bigg)\nonumber\\
=&\sum_{xy=m}\bigg(\sum_{x'\in \varphi^{-1}(x)}f(x')\bigg)\alpha_x\bigg(\sum_{y'\in \varphi^{-1}(y)}g(y'))\bigg)\nonumber\\
=&\sum_{xy=m}\sum_{x'\in \varphi^{-1}(x)}\sum_{y'\in \varphi^{-1}(y)}\bigg(f(x')\alpha_{x'}(g(y'))\bigg).
\end{align*}
Since $\varphi^{-1}$ is a homomorphism, $\varphi^{-1}(x)\varphi^{-1}(y)=\varphi^{-1}(xy)$ and so $\varphi^{-1}(m)=\varphi^{-1}(x)\varphi^{-1}(y)$. So
\begin{align}\label{B}
(f^*g^*)(m)=\sum_{xy=m}\sum_{\substack{m'=x'y'\\m'\in \varphi^{-1}(m)}}\bigg( f(x')\alpha_{x'}(g(y'))\bigg).
\end{align}
By equations \ref{A} and \ref{B} we see that $(fg)^*(m)=(f^*g^*)(m)$. We have also
\begin{align*}
(f+g)^*(x)=&\sum_{x'\in \varphi^{-1}(x)}(f+g)(x')=\sum_{x'\in \varphi^{-1}(x)}(f(x')+g(x'))\nonumber\\
=&\sum_{x'\in \varphi^{-1}(x)}f(x')+\sum_{x'\in \varphi^{-1}(x)}g(x')=f^*(x)+g^*(x).
\end{align*}
Thus $\varphi^*:R[[N,\omega,\leq]]\rightarrow R[[M,\omega,\leq]]$ is a ring homomorphism.
Now, we show that $\varphi^*$ is surjective. Suppose that $f\in R[[M,\omega,\leq]]$, where $\{f(n)\}_{n\in M}$ are the coefficients of $f$ in $R$. For every $n\in M$, the set  $\varphi^{-1}(n)$ is nonempty and finite, say $\varphi^{-1}(n)=\{m_1,m_2,\dots ,m_k\}$, where $k$ and all the $m_{i}$ depends on $n$. We define the function $g\in R[[N,\omega,\leq]]$ as follows
\begin{align}
g(m_j)=\begin{cases}
f(n)\quad ;   j=1 \\
0 \quad  ;   \text{otherwise}.
\end{cases}
\end{align}
Notice that $g$ is independent of $n$, since if $n \neq n'$, then $\varphi^{-1}(n) \cap \varphi^{-1}(n')=\emptyset$. Also, for each $n\in M$ we have
\begin{align*}
\varphi^{*}(g)(n)=\sum_{m\in \varphi^{-1}(n)}g(m)=\sum_{j=1}^k g(m_j)=g(m_1)=f(n).
\end{align*}
This means that $\varphi^{*}(g)=f$, and hence $\varphi^*$ is surjective. So we proved the claim.
It is well-known that there is an strict monoid surjection
 $\varphi :(\mathbb{F}^n,\preccurlyeq)\rightarrow (M,\preccurlyeq)$ for some $n\in \mathbb{N}$.
Also, the identity map $(M,\preccurlyeq)\to (M,\leq)$ is a surjection.
So the composition of these two maps is a surjection and by \cite[Lemma 2.1]{brook}. Hence $R[[M,w,\leq]]$ is a homomorphic image of the ring $R[[\mathbb{F}^n,\omega, \preceq]]$. Since $R[[\mathbb{F}^n,\omega,\preceq]]=R[[\mathbb{F}^n,\omega]]$ and $R[[\mathbb{F}^n,\omega]]$ is Noetherian, its projection $R[[M,w,\leq]]$ is also Noetherian.
Moreover, we show that $R[[M,\omega,\leq]]=R[[M,\omega]]$. If $R[[M,\omega,\leq]]$ is left Noetherian, then $\Downarrow(M,\preccurlyeq)$ is Artinian.
By applying \cite[Lemma 2.1(2)]{brook} to the identity map $(M,\preccurlyeq) \to (M,\leq)$, one can see that $\Downarrow(M, \leq)$ is Artinian.
Thus $R[[M,\omega, \leq]]=R[[M,\omega]]$.

$\Rightarrow$) The method of this part is inspired from \cite[Theorem 4.3]{brook}. The trivial case of $M$ is obvious.   By \cite[Lemmas 3.1 and 3.2]{brook}, $M$ is strict
and $\preccurlyeq$ is a partial order on $M$.

Suppose $T=R[[M,\omega,\leq]]$ is left Noetherian.
One can see that $M$ is finitely generated similar to the proof of
\cite[Theorem 4.3]{brook}. Hence we have to prove
that $R$ is Noetherian similar to the proof of (\cite[Theorem 5.2(i)]{riben},
\cite[Theorem 3.1(i)]{vara}). Let  $I_T=\{f\in T\mid \omega_x(f(y))\in I ; x,y\in M\}$.
It is easy to see that $I_T$ is a left ideal of $T$. So for each ideal $I$ of $R$,
there is a correspondent ideal in $T$. Also if $I \subset J$ ,then
$I_T\subset J_T$. Hence if there exists a nonstabilized ascending
chain in $R$, then there is one in $T$. But this is impossible, so $R$ is left Noetherian.
\end{proof}

In Theorem \ref{go} if we set $\sigma$ the identity homomorphism then we have:
\begin{corollary}
\emph{\cite[Theorem 4.3]{brook}}
Let $R$ be a ring and $(M,\le)$ a positive strictly ordered monoid.
Then $R[[M,\le]]$ is left Noetherian if and only if $R$ is left Noetherian
and $M$ is finitely generated.
\end{corollary}

Finally, we conclude the following result which connects the results of previous sections.
\begin{corollary}
Let $R$ be an $S$-Noetherian von Neumman regular ring and  $S$ a denominator set. Assume that  $(M,\leq)$ is a finitely generated positive strictly ordered monoid and $\omega_{m}$ an automorphism of $R$ with $\omega_m\omega_n=\omega_n\omega_m$ for each $m,n\in M$. Then $(S^{-1}R)[[M,\omega]]$ is a left Noetherian ring.
\end{corollary}
\begin{proof}
The ring $S^{-1}R$ is Noetherian  by Theorem \ref{tie}. Since $(M,\leq)$ is a positive strictly ordered monoid and $\omega_{m}$ is an automorphism for all $m\in M$, $(S^{-1}R)[[M,\omega]]$ is a Noetherian ring by Theorem \ref{go}.
\end{proof}

\section{S-Noetherian property of generalized skew power series rings}

Recall that a ring is called right \emph{duo} (resp., left duo) if all of its right (resp., left) ideals are two-sided.
Also, a right and left duo ring is called a duo ring.
We know that if a ring is duo, then every prime ideal is completely
prime.
It is known that a power series ring over a duo ring need not be duo (on either side).

\begin{lemma}\label{duo}
Let $R$ be a duo ring and $S\subset R$  a denominator set. If $s\in S$, $r\in R$ then there exists  $s_1\in S$ such that $srs_1=rss_1$.
\end{lemma}

\begin{proof}
Let $s\in S$ and $r\in R$. Since $R$ is duo, there exist $s'\in S$
such that $sr=rs'$, so
$\frac{1}{s}.\frac{sr}{1}=\frac{1}{s}.\frac{rs'}{1}$. Hence
$\frac{r}{1}=\frac{rs'}{s}=\frac{r}{1}.\frac{s'}{s}$.
Thus $\frac{r}{1}(1-\frac{s'}{s})=0$, which means that
$\frac{r(s-s')}{s}=0_{S^{-1}R}$. So $r(s-s')s_1=0_{R}$.
So $rss_1=rs's_1$ and since $rs'=sr$ we have $srs_1=rss_1$.
\end{proof}

In the previous result, it is easy to see that if $s\in S$, $r\in R$, then there exists  $s_1\in S$ such that $s_1sr=s_1rs$.
We will use this point in the  proposition below.

\begin{proposition}\label{prime ideal}
Let $R$ be a duo ring, $S\subseteq R$ a denominator set and
$M$ an $S$-finite $R$-module. Then $M$ is $S$-Noetherian if
and only if  $PM$ is an $S$-finite submodule, for
each $S$-disjoint prime ideal $P$ of $R$.
\end{proposition}

\begin{proof}
The ``only if" part is clear. For the converse, assume that $PM$
is $S$-finite for each $P$ prime ideal of $R$ with $P\cap S=\emptyset$.
Since $M$ is $S$-finite, $wM\subseteq F$ for some $w\in S$ and some
finitely generated submodule $F$. If $M$ is not $S$-Noetherian, the set
$\mathfrak {F}$ of all non-$S$-finite submodules of $M$ is not empty.
So $\mathfrak {F}$ has a maximal element like $N$ by Zorn's lemma.
We claim that $P=[N:M]:=\{r\in R\mid rM\subseteq N\}$ is a prime ideal of $R$ and is disjoint from $S$.
Suppose to the contrary that $P\cap S\neq \emptyset$ and  $s\in P\cap S$. Then we have
\begin{align*}
swN \subseteq swM\subseteq sF\subseteq sM\subseteq N.
\end{align*}
So $swN\subseteq sF\subseteq N$ and $N$ becomes $S$-finite. This contradiction
shows that $P\cap S=\emptyset$. Now suppose that $P$ is not a prime ideal of
$R$. So $P$ is not completely prime. So there exist $a,b\in R\setminus P$
and $ab\in P$. So $N+aM$ is $S$-finite, hence $s(N+aM)\subseteq (R(n_1+am_1)+ \cdots+ R(n_p+am_p))$
for some $s\in S$, $n_i\in N$ and $m_i \in M$. Also $[N:a]$ is $S$-finite. So
$t[N:a]\subseteq (Rq_1+Rq_2+\cdots+Rq_k)$ for some $t\in S$ and
$q_j\in [N:a]$. Since $R$ is duo and $S$ is a denominator set in $R$,   there exists $s''\in S$ such that $s''at=s''ta$ by Theorem \ref{duo}. Also $s(N+aM)\subseteq (R(n_1+am_1)+\cdots+R(n_p+am_p))$.
Thus $sx=\sum r_in_i+r_iam_i$. This means that $sx=\sum r_in_i+a\sum r'_im_i$ for
some $r'_i\in R$. Since $sx,\sum r_in_i\in N$, we have $\sum r'_im_i\in [N:a]$. So
\begin{align*}
s''tsx=s''t\sum r_in_i+s''t\sum ar'_im_i=\sum s''tr_in_i+s''at\sum r'_im_i=\sum s''tr_in_i+s''a\sum c_jq_j.
\end{align*}
So $s''tsx=\sum s''tr_in_i+\sum c's''_jaq_j$ for some $c'_j\in R$.
Hence $s''tsx\in (Rn_1+\cdots+Rn_p+Rs''aq_1+\cdots+Rs''aq_k)$.
So $s''tsN\subseteq (Rn_1+\cdots +Rn_p+Rs''aq_1+\cdots+Rs''aq_k) \subseteq N$.
Thus $N$ is $S$-finite and this contradicts to the fact that $N$ is maximal in $\mathfrak {F}$.
Therefore $P$ is a prime ideal of $R$.
Moreover  $P=[N:M]\subseteq [N:F]\subseteq [N:wM]=[P:w]= P$. Hence $[N:F]=P$.
Let $F=(Rf_1+Rf_2+\cdots+Rf_k)$. Since $R$ is a duo ring, $P=[N:\sum Rf_i]= \bigcap[N:f_i]$.
So $P=[N:f_i]$ for some $f_i\in\{f_1,f_2,\cdots ,f_k\}.$
One can show that $tN\subseteq (Rn_1+Rn_2+\cdots+Rn_l)+PM$
for some $t\in S$ and $n_i\in N$ as above or in similar way as that employed in \cite[Proposition 4]{anderson}.
Since $PM$ is $S$-finite,  $vPM\subseteq G\subseteq PM\subseteq N$
for some $v\in S$ and a finitely generated submodule $G$ of $M$. So
\begin{align*}
vtN\subseteq v(Rn_1+Rn_2\cdots +Rn_l)+vPM \subseteq (Rn'_1+Rn'_2\cdots +Rn'_l)+G\subseteq N
\end{align*}
for some $n'_i\in N$. So $N$ becomes $S$-finite which is a contradiction. So $M$ is $S$-Noetherian.
\end{proof}

\begin{lemma}
Let $R$ be a ring with an endomorphism $\sigma$. If $R[[x;\sigma]]$ is a duo ring,  then $\sigma$ is surjective.
\end{lemma}

\begin{proof}
Suppose that $a\in R$. Since $R[[x;\sigma]]$ is a duo ring we have $ax=xf$ such that $f=\sum_{i=0}^{\infty}f_ix^i$. So
$xf=x\sum_{i=0}^{\infty}f_ix^i=\sum_{i=0}^{\infty}\sigma(f_i)x^{i+1}$. Now, since $ax=xf$, $\sigma(f_i)=0$ for all $i \neq 0$ and $\sigma(f_0)=a$. Thus, for each $a\in R$ there exists $f_0\in R$ such that $a=\sigma(f_0)$.
\end{proof}

\begin{theorem}\label{power}
Let $R$ be a ring, $S\subseteq R$ a $\sigma$-anti-Archimedean
denominator set \emph{(}consisting nonzero devisors\emph{)} and $\sigma_1,\cdots,\sigma_n$ are  monomorphisms of $R$ with $\sigma_i\sigma_j=\sigma_j\sigma_i$, for each $i,j$. Assume that   $R[[X_1,\cdots ,X_n;\sigma_1,\cdots,\sigma_n]]$ is a duo ring.
If $R$ is $S$-Noetherian, then the ring $R[[X_1,\cdots ,X_n;\sigma_1,\cdots,\sigma_n]]$ is also $S$-Noetherian.
\end{theorem}

\begin{proof}
We use the method   in \cite[Proposition 10]{anderson} employed by   Anderson and Dumitrescu. As $S$ is $\sigma$-anti-Archimedean in every ring containing $R$ as a subring,
we shall prove the case $n=1$, so we assume that  $T=R[[x;\sigma]]$ is duo and
$\sigma$ is an automorphism of $R$. It is enough to prove that every prime ideal $P$ of $T$ is $S$-finite. Let
$\pi:T\rightarrow R$ the $R$-algebra homomorphism sending $x$ to zero and $P'=\pi(P)$.
Since $R$ is $S$-Noetherian, there exists $s\in S$ such that $sP'\subseteq (Rg_1(0)+Rg_2(0)+ \cdots +Rg_k(0))$
for some $g_i\in P$. If $x\in P$, then $P=( TP'+Tx)$. If $g_i(x)=\sum a_ix^i$, then
$g_i(x)=\sum x^i\sigma^{-i}(a_i)\in (TP'+Tx)$. So
$sP\subseteq (TP'+Tx)= (Tg_1+\cdots +Tg_k) \subseteq P$.
This means that $P$ is $S$-finite. Let $x\notin P$ and $f\in P$. So
$sf(0)=\sum d_{0,j}g_j(0)$ for some $d_{0,j}\in R$. So $xf_1=sf-\sum d_{0,j}g_j \in P$
for some $f_1\in T$. Considering $x\notin P$, $f_1\in P$. So $sf_1=\sum d_{1,j}g_j+xf_2$ for some $f_2\in T$.
Hence $\sigma(s)sf=\sum \sigma(s) d_{0,j}g_j+x\sum d_{1,j}g_j+x^2f_2$. Also $f_2\in P$,
since $x\notin P$ and $sf_1-\sum d_{1,j}g_j \in P$. In this way, one can see that for each $L\ge 0$,
\begin{align*}
\big(\prod_{l=0}^{L}\sigma^l(s)\big)f=\sum_{i=0}^{L}x^i\sum_{j=1}^k (\prod_{l=i+1}^{L}\sigma^l(s))d_{i,j}g_j+x^{L+1}f_{L+1}.
\end{align*}
Since $S\cap \big(\bigcap_{l\geq 1,i_j\in \mathbb{N}\cup \{0\}}\sigma^{i_1}(s)\cdots \sigma^{i_l}(s)R\big)\neq \emptyset$,
 there exists $t\in R$ such that $\frac{t}{\sigma^{i_1}(s)\cdots \sigma^{i_k}(s)}\in R$ for
each $i_j\in \mathbb{N}\cup \{0\}$, $k\in \mathbb{N}$. Moreover
\begin{align*}
tf=\sum_j \sum_i \Big(\frac{ts\sigma^{-i}(d_{ij})}{\prod_l\sigma^l(s)} \Big)x^ig_j.
\end{align*}
So $tf=\sum_j h_jg_j$ where $h_j=\sum_i \frac{ts\sigma^{-i}(d_{i,j})}{\prod_l\sigma^l(s)}x^i$. So
$tf\in (Tg_1+Tg_2+\cdots +Tg_k)$. Hence $tP\subseteq (Tg_1+Tg_2+\cdots +Tg_k)$.
Since $g_i\in P$, $(Tg_1+Tg_2+\cdots +Tg_k) \subseteq P$. Thus
$R[[x;\sigma]]$ is an $S$-Noetherian ring.
\end{proof}

The following proposition which is proved in \cite{anderson}, is the corollary of the above theorem.

\begin{corollary}\emph{\cite[Proposition 10]{anderson}}
Let $R$ be a commutative ring and $S\subseteq R$ an anti-Archimedean
multiplicative set of $R$. If $R$ is $S$-Noetherian, then so is $R[[X_1,\cdots ,X_n]]$.
\end{corollary}

A ring R is called \emph{strongly regular} if every principal right (or left) ideal is generated by a central idempotent.
A ring is said to be left \emph{self injective} if it is injective as a left module over itself.
Hirano in \cite[Theorem 4]{hirano} shows that if $R$ is a self-injective
strongly regular ring, then $R[[x]]$ is a duo ring. \\

We have the following generalization of a theorem of D.D. Anderson and Tiberiu Dumitrescu \cite[Proposition 10]{anderson}.

\begin{theorem}
Let $R$ be a duo ring with an automorphism $\sigma$ and $S\subseteq R$ a $\sigma$-anti-Archimedean
denominator set \emph{(}consisting nonzero devisors\emph{)}.
If $R$ is $S$-Noetherian, then so is the skew power series ring $R[[x;\sigma]]$.
\end{theorem}

\begin{proof}
We can prove this theorem in a similar way as in Theorem \ref{power}. Consider the notations in the proof of Theorem \ref{power}. Let $x\in P$. Since $\sigma$ is bijective, $P$ is $S$-finite. Let $x\notin P$ and $f\in P$, so $xf_1=sf-\sum d_{0i}g_i \in P$. Note that for each $h\in R[[x;\sigma]]$ and $I$  is a left ideal of $R[[x;\sigma]]$, $xh\in I$ yields that $xR[[x;\sigma]]h\in I$. So $f_1\in P$. The rest of the proof is similar to what we did in Theorem \ref{power}.
\end{proof}

The following corollary is a generalization of the case $n=1$ in \cite[Proposition 10]{anderson} for the category of duo rings.

\begin{corollary}
Let $R$ be a duo ring and $S\subseteq R$ an anti-Archimedean
denominator set \emph{(}consisting nonzero devisors\emph{)} of $R$.
If $R$ is $S$-Noetherian, then  so is the  power series ring $R[[x]]$.
\end{corollary}

Now we  extend  the last result for the skew generalized power series ring $R[[M,\omega]]$.

\begin{theorem}
Let $R$ be a duo ring,  $(M,\le)$  a positive
strictly ordered commutative monoid and $\omega_{m}$  a monomorphism of $R$ with $\omega_m\omega_n=\omega_n\omega_m$ for each $m,n\in M$. Assume that $S\subset R$ is an $\omega_m$-anti-Archimedean
denominator set \emph{(}consisting nonzero devisors\emph{)} of $R$ and  $R[[M,\omega]]$ be a duo ring. Then
  $R[[M,\omega]]$ is left (or right) $S$-Noetherian if and only if $R$ is left (or right) $S$-Noetherian and $M$ is finitely generated.
\end{theorem}

\begin{proof}
($\Leftarrow$) We use the method of G. Brookfeild employed in \cite{brook}. We know that the surjective homomorphism $\varphi:\mathbb{F}^n\longrightarrow M$ (where $\mathbb{F}$ is a free monoid)
induces a projection $$\varphi^*:R[[\mathbb{F}^n,(\omega,\preceq)]]\longrightarrow R[[M,(\omega,\le)]]$$
and $R[[M,(\omega,\le)]]=R[[M,\omega]]$ by \cite[Theorem 4.3]{brook}.
Moreover, since $R[[\mathbb{F}^n,(\omega,\preceq)]]$ is $S$-Noetherian, so is $R[[M,\omega]]$ by \cite[Lemma 2.2]{liu} for noncommutative version.

($\Rightarrow$) Let $A:=R[[M,\omega]]$ be $S$-Noetherian. Let $\{m_n| n\in \mathbb{N}\}$
be an infinite sequence in $M$. Let $I=(Ae_{m_1}+Ae_{m_2}+\cdots)$.
Since $A$ is $S$-Noetherian, there exists $s\in S$ such that $c_sI\subseteq J\subseteq I$
for $J$ finitely generated ideal of $A$. So $c_sI\subseteq (Ae_{m_{i_1}}+Ae_{m_{i_2}}+\cdots+Ae_{m_{i_k}})$
for some $k\in \mathbb{N}$. So $c_se_{m_l}=\sum_{t=0}^k f_t e_{m_{it}}$ for
some $l\neq i_t$. So $m_l\in \bigcup_{t=0}^k \supp(f_te_{m_{i_t}})$ for each $m\in M$,
$(f_te_{m_{i_t}})(m)=\sum_{m'm''=m}f_t(m')\omega_{m'}(e_{m_{i_t}}(m''))$.
So $m_{l}\in \bigcup_{t=0}^k \big\{\supp(f_t)+\supp(\omega_{m'}(e_{m_{i_t}}(m'')))\big\}$.
There exists $m_1\in M$ such that $m_1m_{it}=m$ for some $0\le t \le L$. So
$$(f_te_{m_{i_t}})(m)=f_t(m_1)\omega_{m_1}(e_{m_{i_t}}(m_{i_t}))=f_t(m_1).$$
Thus $m_1\in \supp(f_t)$ and $m_1m_{i_t}\in \supp(f_te_{m_{i_t}})$ for some $0\le t\le L$.
So for each $m\in \supp(\omega_{m'}(e_{m_{i_t}}(m')))$, $m_{i_t}\preceq m$ for some
$0\le t \le L$. Since $m_l \in \supp(f_te_{m_{i_t}})$, $m_{i_t}\preceq m_l$ for some $0\le t \le L$.
Since $M$ is positive strictly ordered monoid, $M$ is finitely generated by \cite[Lemma 3.3]{brook}.

Let $I$ be an ideal of $R$, so $AI$ is an ideal of $A$. So there exists $s\in S$ such that
$c_sAI\subseteq J\subseteq AI$ for some $J$ finitely generated ideal of $A$.
Set
\begin{align*}
T=\{f(\pi(f))| f\in c_sAI\}.
\end{align*}
We claim that $T=sI$. Let $t\in T$, so $t=h(\pi(h))$ and $h=c_sg$
for some $g\in AI$. So $t=sg(\pi(sg))$. This means that $t\in sI$
considering the fact that
\begin{align*}
I=\{f(\pi(f))| f\in AI\}.
\end{align*}
So $T\subseteq sI$. Now let  $i\in I$, so $i\in AI$.
Since $si(m)=0$ for $m\neq 1$,  $si\in T$. Thus $sI\subseteq T$.
Hence $sI=T$. But $sI=T\subseteq J'\subseteq I$ where $J'=\{f(\pi(f))|f\in J\}$.
Let $J=(Aj_1+Aj_2+\cdots+Aj_p)$. So it is easy to show that
$J'=(Rj_1(\pi(j_1))+Rj_2(\pi(j_2))+\cdots+Rj_p(\pi(j_p)))$.
So $J'$ is finitely generated in $R$. Hence $I$ is $S$-finite and $R$ is left $S$-Noetherian.
\end{proof}
Recall from \cite{Brewer}, that  a ring $R$
is right (left) $\aleph_0$-injective provided any homomorphism from a countably
generated right (left) ideal of $R$ into $R$ extends to a right (left) $R$-module
endomorphism of $R$. By an $\aleph_0$-injective ring we mean a right and left
$\aleph_0$-injective ring.
\begin{corollary}
Let $R$ be an strongly regular and an $\aleph_0$-injective ring with automorphisms $\sigma_1,\sigma_2$
such that $\sigma_1\sigma_2=\sigma_2\sigma_1$. Assume that  $S\subset R$ is an $\sigma_1$, $\sigma_2$ anti-Archimedean  denominator set \emph{(}consisting nonzero devisors\emph{)}. If $R$ is left (or right) $S$-Noetherian, then so is $R[[x,y;\sigma_1,\sigma_2]]$.
\end{corollary}

\begin{proof}
Assume that $R$ is an strongly regular and $\aleph_0$-injective ring. Then by \cite{zhim}, $A=R[[x;\sigma_1]]$ is duo ring and $S$-left Noetherian ring. Then $A[[y;\sigma_2]]$ is left $S$-Noetherian ring.
\end{proof}

The following corollary is a generalization of the case $n=2$ of in \cite[Proposition 10]{anderson} for the category of duo rings.

\begin{corollary}
Let $R$ be an strongly regular  self-injective ring and $S\subseteq R$ an anti-Archimedean denominator set \emph{(}consisting nonzero devisors\emph{)} of $R$. If $R$ is left (or right) $S$-Noetherian, then so is $R[[x,y]]$.
\end{corollary}

\begin{corollary}
Let $R$ be a duo ring, $S\subseteq R$ an anti-Archimedean
denominator set \emph{(}consisting nonzero devisors\emph{)} of $R$. Assume that $R[[M]]$  is a duo ring. Then  $R[[M]]$ is left (or right) $S$-Noetherian
 if and only if  $R$ is left (or right) $S$-Noetherian and $M$ is finite generated.
\end{corollary}

\end{document}